\documentclass{article}
\usepackage{amssymb,amsmath}
\usepackage{graphicx,enumerate}
\usepackage[font=small,labelfont=bf]{caption}
\usepackage[T1]{fontenc}
\usepackage[utf8]{inputenc}
\usepackage[all]{xy}

\title{Inventory problems and the parametric measure $m_{\lambda}$}

\author{Irina Georgescu$^\ast$ \\ \footnotesize Bucharest University of Economics\\ \footnotesize Department of Informatics and Economic Cybernetics\\ \footnotesize Pia$\c{t}$a Romana No 6  R 70167, Oficiul Postal 22, Bucharest, Romania\\
 \footnotesize Email: irina.georgescu@csie.ase.ro\\
 \footnotesize $^\ast$ Corresponding author  }

\date{}

\begin{document}
\maketitle

\begin{abstract}
The credibility theory was introduced by B. Liu as a new way to describe the fuzzy uncertainty. The credibility measure is the fundamental notion of the credibility theory.
Recently, L.Yang and K. Iwamura extended the credibility measure by defining the parametric measure $m_{\lambda}$ ($\lambda$ is a real parameter in the interval $[0,1]$
and for $\lambda= 1/2$ we obtain as a particular case the notion of credibility measure).

By using the $m_{\lambda}$- measure,  we studied in this paper a risk neutral multi-item inventory problem.
Our construction generalizes the credibilistic inventory model developed by Y. Li and Y. Liu in 2019.
In our model, the components of demand vector are fuzzy variables and the maximization problem is formulated by using the notion of $m_{\lambda}$--expected value.

We shall prove a general formula for the solution of optimization problem, from which we obtained
effective formulas for computing the optimal solutions in the particular cases where the demands are trapezoidal and triangular fuzzy numbers.
For $\lambda=1/2$ we obtain as a particular case the computation formulas of the optimal solutions of the credibilistic inventory problem of Li and Liu. These computation formulas are applied for some $m_{\lambda}$- models obtained from numerical data.

\end{abstract}


\newtheorem{definitie}{Definition}[section]
\newtheorem{propozitie}[definitie]{Proposition}
\newtheorem{remarca}[definitie]{Remark}
\newtheorem{exemplu}[definitie]{Example}
\newtheorem{intrebare}[definitie]{Open question}
\newtheorem{lema}[definitie]{Lemma}
\newtheorem{teorema}[definitie]{Theorem}
\newtheorem{corolar}[definitie]{Corollary}

\newenvironment{proof}{\noindent\textbf{Proof.}}{\hfill\rule{2mm}{2mm}\vspace*{5mm}}

\section{Introduction}

Let's consider a company that produces several types of goods (items). It will be assumed that buyers can order in advance. An inventory problem is a mathematical model that describes the management of this company.

There are mathematical models in which the management activity of the company is carried out in a single period and models with several periods.  The inventory models can also be classified according to the attitude towards risk of a decision-maker: there are models in which the decision maker has a risk-averse attitude and models in which his attitude is neutral.

The mathematical formulation of the inventory model starts from the following initial data (model parameters) : $c_1, \ldots, c_n$ are unit fixed costs per inventoried item, $d_1, \ldots, d_n$ are unit revenues per inventoried item and $h_1, \ldots, h_n$
are unit holding costs per inventoried item.

The demands are mathematically modeled by the variables $D_1, \cdots, D_n$; the order quantities will be the variables $x_1, \cdots, x_n$.
The total profit from the sale of the $n$ types of goods will have the following expression::
$\pi(\vec{x}, \vec{D}) =\displaystyle \sum_{i=1}^n (d_i x_i-c_i -\frac{h_i x_i^2}{2 D_i})$ (see [23], [18]).

In a neutral inventory problem, one will determine those values of $x_1,\cdots,x_n$ for which the total profit $\pi(\vec{x}, \vec{D}) =\displaystyle \sum_{i=1}^n (d_i x_i-c_i -\frac{h_i x_i^2}{2 D_i})$ is maximal. When making the decision the risk is taken into account, the values of $x_1,\cdots,x_n$ will be determined so that at the same time the maximum profit is achieved, and the risk (represented by various mathematical concepts) to be minimal.

The formulation of an inventory problem depends on how the demands $D_1,\cdots,D_n$ are modeled, as well as on how profit maximization and risk minimization are evaluated. The classical treatment of inventory problems is a probabilistic one: the demands $D_1, \cdots, D_n$ are random variables and, for risk - neutral models,
the objective function of the maximization problem is the expected value of the total profit. In the case of a risk - averse attitude of the decision maker, several ways to describe the risk were proposed: in \cite{chen}, \cite{choi} by means of mean-variance models and in \cite{luciano} by using the value-at-risk (VaR) as a risk measure.
In \cite{ahmed}, the coherent risk measures \cite{artzner} have been used in defining the objective function of an inventory problem. Using  the multi-item inventory system introduced by Luciano et al. \cite{luciano} (called, shortly, LCP-model), \cite{borgonovo} developed several inventory problems, with decision-makers having various positions towards risk: from a neutral attitude to risk-averse attitude,
corresponding to variance, mean-absolute deviation (MAD) and conditional value-at-risk (CVaR) as risk measures.

The credibility theory, specially developed by Liu in \cite{liu1}, is another way to model the fuzzy uncertainty. Its fundamental concept is the credibilistic measure \cite{liu2}
and its main indicators are the credibilistic expected value and the credibilistic variance (cf. \cite{liu1}, \cite{liu2}).
From the literature dedicated to the credibilistic treatment of inventory problems we mention the papers: \cite{ghasemyYaghin}, \cite{guo1}, \cite{kumar}, \cite{mittal}. In this paper we will have as the starting point the papers \cite{li1}, \cite{li2}, \cite{li3} of Li and  Liu:
the first one concerns a multi - item inventory problem in which the decision - maker is neutral and the second one
is a risk - averse inventory model. In both papers, the demands and the total profits are fuzzy variables and the expected profit is
 the credibilistic expected value of total profit. In \cite{li2} appears a risk evaluated by the notion of absolute semi - deviation.

In \cite{YI}, Yang and Iwamura introduced a new measure $m_{\lambda}$ as a convex linear combination of a possibility measure $Pos$ and
its associated necessity measure $Nec$ ($\lambda$ is a parameter in the interval $[0,1]$). By using the measure $m_{\lambda}$, in \cite{dzuche1}
the notions of the expected value $E_{\lambda}(\xi)$ and the variance $Var_{\lambda}(\xi)$ of a fuzzy variable $\xi$ are defined. These two indicators retain
some algebraic properties of the possibilistic indicators corresponding to \cite{liu1}. In this way, the credibility theory is enlarged
to a new theory that models the fuzzy uncertainty (this will be named $m_{\lambda}$ - theory). An issue that arises  naturally is an $m_{\lambda}$--theory leading to the development of different economic and financial themes.
Papers \cite{kinnunen}, \cite{kinnunen1}, \cite{georgescu2},  \cite{georgescu3} introduce new credibilistic real options models, which are based on the optimism–-pessimism measure and interval--valued fuzzy numbers. The model outcomes are compared to the original credibilistic real options model through a numerical case example in a merger and acquisition context.
Paper \cite{dzuche1} applies $m_{\lambda}$--theory in the study of optimal portfolios when assets returns are described by triangular or trapezoidal fuzzy variables.

In this paper we shall study a multi - item risk neutral inventory problem in the framework of an $m_{\lambda}$ - theory. We shall assume that the demands $D_1,\cdots,D_n$ are fuzzy variables and the criterion used in determination of the order quantities $x_1,\cdots,x_n$ is the maximization of the $m_{\lambda}$ - expected value $\displaystyle \sum_{i=1}^n [d_i x_i-c_i -\frac{h_i x_i^2}{2 }E_{\lambda}(\frac{1}{D_i})]$ of the total profit.
We shall prove a general formula for computing the solution of optimization problem, of which we will then get formulas for effective
computation of inventory problem solution whenever the demands are trapezoidal or triangular fuzzy numbers. For $\lambda = \frac{1}{2}$ we shall obtain as a particular case
the credibilistic inventory problem of \cite{li1}, as well as the form of its solution.

The paper is structured as follows. Section $2$ contains introductory material on possibility and necessity measures, credibility measure and $m_{\lambda}$ - measure, as well as on their relationship. In Section $3$ we present the
definition of the $m_{\lambda}$ - expected value and some of its basic properties. Section $4$ deals with the construction of a risk neutral inventory model whose objective function is defined  by
using the notion of $m_{\lambda}$ - expected value.
By using the linearity of $m_{\lambda}$--expected operator $E_{\lambda}(.)$, a general formula for the solution of the maximization problem is obtained.
In Section $5$ we proved some explicit formulas for this solution in the particular cases when the demands are trapezoidal and triangular fuzzy numbers. The proofs of these
formulas are based on the form of $m_{\lambda}$--expected value $E_{\lambda}(A)$ of a trapezoidal fuzzy number $A$ (see Proposition 3.4).  Section $6$  highlights how
by applying the percentile method of Vercher et al. \cite{vercher} we can build an inventory problem starting from a dataset. In this inventory problem, the components of a demand vector
are trapezoidal fuzzy numbers, such that one can apply the formulas from Section $4$ to compute the solution of the optimization problem.


\section{Preliminaries}

Let $X$ be a universe whose elements can be individuals, objects, states, alternatives, etc. An events $A$ is a subset of $X$: the set of events will be the family $P(X)$ of
subsets of $X$. The complement of the event $A$ will be denoted by $A^c$.

In this paper we shall assume that the elements of the universe $X$ are real numbers ($X\subseteq \mathbb{R}$). A fuzzy variable will be an arbitrary function $\xi:X\rightarrow \mathbb{R}$.

The notions of possibility measure and necessity measure can be introduced both axiomatically and through  a possibility distribution (cf. \cite{zadeh}, \cite{dubois}, \cite{georgescu}).

A possibility measure on $X$ is a function $Pos:\mathcal{P}(X) \rightarrow [0,1]$ such that

($Pos1$) $Pos(\emptyset) = 0$; $Pos(X) = 1$;

($Pos2$) $Pos(\bigcup_{i\in I}A_i) = sup_{i\in I}Pos(A_i)$, for any family $(A_i)_{i\in I}$ of events.

A necessity measure on $X$ is a function $Nec:\mathcal{P}(X)\rightarrow [0,1]$ such that

($Nec1$) $Nec(\emptyset) = 0$; $Nec(X) = 1$;

($Nec2$) $Nec(\bigcap_{i\in I}A_i) = inf_{i\in I}Nec(A_i)$, for any family $(A_i)_{i\in I}$ of events.

The notions of possibility measure and necessity measure are dual: to each possibility measure $Pos$ one can assign a necessity measure $Nec(A) = 1 - Pos(A^c)$ and, vice-versa, to each necessity measure $Nec$ one can assign a possibility measure $Pos(A) = 1 - Nec(A^c)$.

Given a possibility measure $Pos$ on the universe $X$, for any parameter $\lambda\in [0, 1]$ consider the function $m_{\lambda}:P(X)\rightarrow [0,1]$ defined by

(2.1) $m_{\lambda}(A) = \lambda Pos(A) + (1-\lambda) Nec(A)$, for any event $A$;

($Nec$ is here the necessity measure associated with $Pos$).

This new measure was introduced by Yang and Iwamura in \cite{YI} as a convex linear combination of $Pos$ and $Nec$ by means of the weight $\lambda$. If $\lambda = \frac{1}{2}$ then one obtains the notion of credibility measure in the sense of Liu's monograph \cite{liu1}:

(2.2) $Cred(A) = \frac{1}{2}(Pos(A) + Nec(A))$, for any event $A$.

A possibilistic distribution on $X$ is a function $\mu:X\rightarrow [0,1]$ such that $sup_{x\in X} = 1$; $\mu$ is normalized if $\mu(x) = 1$ for some $x\in X$.

Let us fix a possibility distribution $\mu:X\rightarrow [0,1]$. Then one can associate with $\mu$ a possibility measure $Pos$ and a necessity measure $Nec$ by taking

(2.3) $Pos(A) = sup_{x\in A}\mu(x)$, for any event $A$;

(2.4) $Nec(A) = inf_{x\in A}\mu(x)$, for any event $A$.

Then for each parameter $\lambda\in [0,1]$, the measure $m_{\lambda}$ defined by (2.1) will have the following form

(2.5) $m_{\lambda}(A) = \lambda sup_{x\in A}\mu(x) + (1-\lambda) inf_{x\in A}\mu(x)$, for any event $A$.

According to \cite{liu1}, we say that the normalized possibility distribution $\mu$ is the membership function associated with a fuzzy variable $\xi$ if for any event $A$ we have

(2.6) $Pos(\xi\in A) = sup_{x\in A}\mu(x)$.

Then the following equalities hold:

(2.7) $Nec(\xi\in A) = inf_{x\in A}\mu(x)$;

(2.8) $m_{\lambda}(\xi\in A) = \lambda sup_{x\in A}\mu(x) + (1-\lambda)inf_{x\in A}\mu(x)$.

\section{The expected value associated with the measure $m_{\lambda}$}

We fix a parameter $\lambda\in [0, 1]$ and assume that $\xi$ is a fuzzy variable, $\mu$ is its membership function and $m_{\lambda}$ is the measure defined in (2.5).

Following \cite{dzuche1}, the expected value of $\xi$ w.r.t. the measure $m_{\lambda}$ is defined by

(3.1) $E_{\lambda}(\xi)$ = $\int^0_{-\infty}[ m_{\lambda}(\xi\geq r) - 1]dr + \int^{\infty}_0 m_{\lambda}(\xi\geq r)dr$.

If $\lambda = \frac{1}{2}$ then one obtains the credibilistic expected value of $\xi$ w.r.t. the credibility measure $Cr$ defined in $(2.2)$:

(3.2) $E_C(\xi)$ = $\int^{\infty}_0 Cr(\xi\geq r)dr - \int_{-\infty}^0Cr(\xi\leq r)dr$.

The previous notion of credibilistic expected value was introduced by Liu and Liu in \cite{liu2}.

The following result shows that the expected operator $E_{\lambda}(\cdot)$ is linear.

\begin{propozitie} \cite{dzuche1} Let $\xi_1,\xi_2$ be two fuzzy variables such that $E_{\lambda}(\xi_1)< \infty$, $E_{\lambda}(\xi_2)< \infty$ and $\alpha,\beta$ are two non - negative real numbers. Then the following hold:

(3.3) $E_{\lambda}(\xi_1 + \xi_2)$ = $E_{\lambda}(\xi_1) + E_{\lambda}(\xi_2)$;

(3.4) $E_{\lambda}(\alpha \xi_1) = \alpha E_{\lambda}(\xi_1)$.

\end{propozitie}

\begin{lema} If $\xi > 0$ then $E_{\lambda}(\xi)$ = $ \int^{\infty}_0 m_{\lambda}(\xi \geq r)dr$ and $E_{\lambda}(\xi)> 0$.

\end{lema}

According to \cite{liu1}, p.73, a trapezoidal fuzzy variable ( = trapezoidal fuzzy number) $\xi = (r_1,r_2,r_3,r_4)$, with $r_1\leq r_2\leq r_3\leq r_4$, is defined by the following membership function:

(3.5)    $\mu_\xi(x) = \left\{
   \begin{array}{ll}
     \frac{x-r_1}{r_2-r_1} & \quad r_1 \leq x \leq r_2, \\
     1 & \quad r_2 \leq x \leq r_3,\\
     \frac{x-r_3}{r_4-r_3} & r_3 \leq x \leq r_4 ,\\
     0  & \quad otherwise.
   \end{array}
  \right.$

If $r_2 = r_3$ then one obtains the triangular fuzzy number $\xi = (r_1,r_2,r_4)$.

\begin{lema}
\cite{dzuche1} For any trapezoidal fuzzy variable $\xi = (r_1,r_2,r_3,r_4)$ we have:

(3.6)  $m_\lambda(\xi \leq x)=\left\{
   \begin{array}{lllll}
     1 & \quad r_4 \leq x,\\
     \frac{\lambda(r_4-x)+x-r_1}{r_4-r_3} & \quad r_3 \leq x \leq r_4, \\
     \lambda & \quad r_2 \leq x \leq r_3,\\
     \frac{\lambda(x-r_1)}{r_2-r_1} & r_1 \leq x \leq r_2 ,\\
     0  & \quad x \leq r_1.
   \end{array}
  \right.$

\end{lema}

\begin{propozitie}
\cite{dzuche1} For any trapezoidal fuzzy variable $\xi = (r_1,r_2,r_3,r_4)$ the expected value $E_{\lambda}(\xi)$ has the form

(3.7) $E_{\lambda}(\xi)$ = $(1-\lambda)\frac{r_1 + r_2}{2} + \lambda\frac{r_3 + r_4}{2}$.

\end{propozitie}

\begin{corolar}
For any triangular fuzzy variable $\xi = (r_1,r_2,r_4)$ the expected value $E_{\lambda}(\xi)$ has the form

(3.8) $E_{\lambda}(\xi)$ = $(1-\lambda)\frac{r_1}{2} + \frac{r_2}{2}+ \lambda \frac{r_4}{2}$.
\end{corolar}

\section{An inventory problem with fuzzy variables as demands}

This section concerns a risk - neutral multi - item inventory problem characterized by the following two hypotheses:

(I) the components of the demand vector are fuzzy variables;

(II) the objective function of the inventory model is defined by using the expected value operator $E_{\lambda}(\cdot)$ introduced in the previous section.

The inventory problem with $n$ items has the following initial data:

$\bullet$ $c_1,\ldots, c_n$ : unit fixed costs per inventoried item;

$\bullet$ $d_1,\ldots, d_n$ : unit revenues per inventoried item;

$\bullet$ $h_1,\ldots, h_n$ : unit holding costs per inventoried item;

$\bullet$ $\vec D = (D_1,\ldots, D_n)$ : fuzzy demand vector in the inventory problem;

$\bullet$ $\vec x = (x_1,\ldots, x_n)$ : order quantity vector in the inventory problem.

The components $D_1,\ldots, D_n$ of $\vec D$ are fuzzy variables. We shall assume that $c_i\geq 0$, $d_i\geq 0$ and $D_i > 0$, for all $i = 1, \ldots ,n$.

\begin{remarca}
The initial data of the possibilistic inventory problem are similar to the probabilistic inventory problems from \cite{luciano}, \cite{borgonovo}, the credibilistic inventory problems from \cite{li1},\cite{li2} and the possibilistic inventory problems from \cite{georgescu1}.

\end{remarca}

 We will further observe that the essential difference between the three types of models lies in the way of choosing the objective function of the optimization problem:
 the models in \cite{luciano}, \cite{borgonovo} use the probabilistic expected value, those in \cite{li1}, \cite{li2} use Liu credibilistic expected value \cite{liu1} and those in
 \cite{georgescu1} use the possibilistic expected value from \cite{carlsson1}.

Starting from above input data we will formulate a risk-neutral problem. Similar with \cite{li1}, p. 132, the quantity $d_i x_i$ is the total revenue of the $i^{th}$ item and the fuzzy variables $\frac{h_i x_i^2}{2}\frac{1}{D_i}$ is the holding cost of the $i^{th}$ item.

We fix a parameter $\lambda\in [0,1]$, so we can use the expected value operator $E_{\lambda}(\cdot)$ defined in (3.1). According to Lemma 3.2, we remark that $E_f(\frac{1}{D_i}) > 0$, for all $i = 1, \ldots ,n$.

The profit function of item $i$ has the following form:

(4.1) $\pi_i(x_i, D_i) = d_ix_i - c_i - \frac{h_ix_i^2}{2}\frac{1}{D_i}$

The total profit function of the possibilistic inventory problem has the following form:

(4.2) $\pi(\vec x, \vec D)$ =  $\displaystyle\sum_{i=1}^n \pi_i(x_i, D_i)$ = $\displaystyle\sum_{i=1}^n (d_ix_i -c_i - \frac{h_ix_i^2}{2}\frac{1}{D_i})$

Then the optimization problem associated with the previous inventory model has the following form:

(4.3)
$\begin{cases}
\displaystyle \max_{\vec x} E_{\lambda}(\pi(\vec x, \vec D))\\
\vec x \geq 0
\end{cases}$

\begin{remarca}
The objective function in the optimization problem (4.3) is the expected value $E_{\lambda}(\pi(\vec x, \vec D))$ of the fuzzy variable $\pi(\vec x, \vec D)$ (w.r.t. the measure $m_{\lambda})$.

\end{remarca}

\begin{remarca}
For $\lambda = \frac{1}{2}$ we obtain as a particular case the credibilistic inventory problem studied in \cite{li1}:

(4.4)$\begin{cases}
\displaystyle \max_{\vec x} E_{\lambda}(\pi(\vec x, \vec D))\\
\vec x \geq 0
\end{cases}$

\end{remarca}

By applying Proposition 3.1 to (4.2), the  expected value $E_{\lambda}(\pi(\vec x, \vec D))$ can be written

(4.5)
$E_{\lambda}(\pi(\vec x, \vec D))$ = $\displaystyle\sum_{i=1}^n [d_ix_i -c_i - \frac{h_ix_i^2}{2} E_{\lambda}(\frac{1}{D_i})]$

hence the optimization problem (4.3) becomes

(4.6)$\begin{cases}
\displaystyle \max_{x_1, \ldots, x_n} \sum_{i=1}^n [d_ix_i -c_i - \frac{h_ix_i^2}{2} E_{\lambda}(\frac{1}{D_i})]\\
 x_i \geq 0, i=1, \ldots, n
\end{cases}$

The decision - maker aims to find the non - negative values $x_1, \ldots ,x_n$ that maximize the expected total profit $E_{\lambda}(\pi(\vec x, \vec D))$.

In particular, setting $\lambda=\frac{1}{2}$ in (4.6) one obtains the credibilistic inventory problem from \cite{li1}.

(4.7)$\begin{cases}
\displaystyle \max_{x_1, \ldots, x_n} \sum_{i=1}^n [d_ix_i -c_i - \frac{h_ix_i^2}{2} E_C(\frac{1}{D_i})]\\
 x_i \geq 0, i=1, \ldots, n
\end{cases}$

\begin{propozitie} The optimization problem (4.6) has the following solution:

(4.8) $x_i^* = \frac{d_i}{h_i E_{\lambda}(\frac{1}{D_i})}$, for $i = 1, \ldots ,n$

\end{propozitie}

\begin{proof}
In order to find the solution of the optimization problem (4.6) we write the first - order condition

$\frac{\partial}{\partial x_i}\displaystyle\sum_{i=1}^n (d_ix_i -c_i - \frac{h_ix_i^2}{2} E_{\lambda}(\frac{1}{D_i}))=0$, for $i = 1,\ldots,n$,

therefore by a simple computation we obtain the equations

(4.9) $d_i - h_i E_{\lambda}(\frac{1}{D_i}) x_i = 0$, for $i = 1, \ldots, n$.

We remind that $E_{\lambda}(\frac{1}{D_i})> 0$ for $i = 1,\ldots,n$. Thus the solution of the optimization problem (3.5) will have the following form

$x_i^* = \frac{d_i}{h_i E_f(\frac{1}{A_i})}$, for $i = 1, \ldots, n$

\end{proof}

\section{Solution form when the demands are trapezoidal fuzzy variables}

According to Proposition 4.1, in order to compute the values $(x_1^\ast, \ldots ,x_n^\ast)$ of the solution of inventory problem (4.6) we need to compute the expected  values $E_{\lambda}(\frac{1}{D_1})$, \ldots, $E_{\lambda}(\frac{1}{D_n})$.
The computation of these expected values depends on the form of the fuzzy variables $D_1, \cdots, D_n$ and in most cases this operation seems to be very difficult.
In this section we solve this problem whenever the demands $D_1, \cdots, D_n$ are trapezoidal or triangular
fuzzy numbers. The formulas obtained for the computation of the optimal solutions $x_1^\ast, \ldots, x_n^\ast$ have simple algebraic forms which makes them very suitable from a computational point of view.

We will fix the parameter $\lambda\in [0,1]$. The following proposition is a key result of this section: the application of the formula (5.1) will lead us to find the form of optimal solutions $x_1^\ast, \ldots, x_n^\ast$.

\begin{propozitie}
Let $D$ be a trapezoidal fuzzy number $D = (r_1,r_2,r_3,r_4)$ such that $0< r_1\leq r_2\leq r_3\leq r_4$ then the expected value $E_{\lambda}(\frac{1}{D})$ has the following form

(5.1) $E_{\lambda}(\frac{1}{D}) = \frac{\lambda}{r_2- r_1}ln\frac{r_2}{r_1} + \frac{1-\lambda}{r_4- r_3}ln\frac{r_4}{r_3}$

\end{propozitie}

\begin{proof} Firstly we observe that the condition $0< r_1$ means $D>0$, hence one obtains $\frac{1}{D}>0$. By using Lemma 3.3 we get the following equalities:

$m_\lambda(\frac{1}{D}\geq r)=m_\lambda(D \leq \frac{1}{r})=\left\{
   \begin{array}{ll}
     1 & \quad r_4 \leq \frac{1}{r}, \\
     \frac{\lambda(r_4-\frac{1}{r})+\frac{1}{r}-r_3}{r_4-r_3} & \quad r_3 \leq \frac{1}{r} \leq r_4,\\
     \lambda & \quad r_2 \leq \frac{1}{r}  \leq r_3,\\
     \frac{\lambda(\frac{1}{r}-r_1}{r_2-r_1} & \quad r_1 \leq \frac{1}{r}  \leq r_2,\\
     0  & \quad \frac{1}{r} \leq r_1.
   \end{array}
  \right.$

which can be written as follows:

(5.2) $m_\lambda(\frac{1}{D}\geq r)=\left\{
   \begin{array}{ll}
     1 & \quad r \leq \frac{1}{r_4}, \\
     \frac{1}{r_4-r_3}[(1-\lambda) \frac{1}{r}+\lambda r_4-r_3]& \quad \frac{1}{r_4} \leq r  \leq \frac{1}{r_3},\\
     \lambda & \quad \frac{1}{r_3}\leq   r  \leq \frac{1}{r_2},\\
     \frac{\lambda}{r_2-r_1} [\frac{1}{r}-r_1]& \quad \frac{1}{r_2} \leq r  \leq \frac{1}{r_1},\\
     0  & \quad \frac{1}{r} \leq 0.
   \end{array}
  \right.$

According to Lemma 3.2 we obtain

(5.3) $E_{\lambda}(\frac{1}{D})$ = $\int^{\infty}_0 m_{\lambda}(\frac{1}{D}\geq r)dr = I_1+I_2+I_3+I_4$

where $I_1,I_2,I_3,I_4$ have the following expressions

$I_1 = \int_0^{\frac{1}{r_4}}dr = \frac{1}{r_4}$

$I_2 = \frac{1}{r_4-r_3}\int_{\frac{1}{r_4}}^{\frac{1}{r_3}}[\lambda(r_4-\frac{1}{r})+ \frac{1}{r}-r_3]dr$ = $\frac{1}{r_4-r_3}[(1-\lambda)\ln\frac{r_4}{r_3} + (\lambda r_4 - r_3)(\frac{1}{r_3}- \frac{1}{r_4})]$

$I_3 = \lambda\int_{\frac{1}{r_3}}^{\frac{1}{r_2}}dr = \lambda(\frac{1}{r_2}-\frac{1}{r_3})$

$I_4 = \frac{\lambda}{r_2-r_1}\int_{\frac{1}{r_2}}^{\frac{1}{r_1}}[\frac{1}{r}-r_1]dr$ = $\frac{\lambda}{r_2-r_1}[ln\frac{r_2}{r_1}-r_1(\frac{1}{r_1}-\frac{1}{r_2})]$

Substituting in (5.3) these values of $I_1,I_2,I_3,I_4$ we get the formula (5.1)

\end{proof}

\begin{corolar}\cite{li1}
Let $D$ be a trapezoidal fuzzy number $D = (r_1,r_2,r_3,r_4)$ such that $0< r_1\leq r_2\leq r_3\leq r_4$ then the credibilistic
expected value $E_C(\frac{1}{D})$ has the following form

(5.4) $E_C(\frac{1}{D}) = \frac{1}{2(r_2- r_1)}ln\frac{r_2}{r_1} + \frac{1}{2(r_4- r_3)}ln\frac{r_4}{r_3}$

\end{corolar}

\begin{proof}
If we take $\lambda = \frac{1}{2}$ in (5.1) then we obtain the formula (5.4).
\end{proof}

\begin{remarca} If in formula (5.1) one takes $r_2 = r_3$ then $D$ is the triangular fuzzy number $D = (r_1,r_2, r_4)$ and the expected value $E_{\lambda}(\frac{1}{D})$ has the following form

(5.5) $E_{\lambda}(\frac{1}{D}) = \frac{\lambda}{r_2- r_1}ln\frac{r_2}{r_1} + \frac{1-\lambda}{r_4- r_3}ln\frac{r_4}{r_2}$

\end{remarca}

If in (5.5) we set $\lambda = \frac{1}{2}$  then we get the formula of the credibilistic expected value $E_{\lambda}(\frac{1}{D})$ from Theorem $2$ of \cite{li1}:

(5.6) $E_C(\frac{1}{D}) = \frac{1}{2(r_2- r_1)}ln\frac{r_2}{r_1} + \frac{1}{2(r_4- r_2)}ln\frac{r_4}{r_2}$

\begin{remarca} Often in literature a trapezoidal fuzzy number $D$ is given under the form $D = (a-\alpha, a, b, b+\beta)$, with $a, b\in \mathbb{R}$ and $\alpha, \beta \geq 0$ (Figure $1$). Thus its membership $\mu_D$ has the form:

$$
\mu_D(x) = \left\{
   \begin{array}{ll}
     1-\frac{a-x}{\alpha} & \quad a-\alpha \leq x \leq a, \\
     1 & \quad a \leq x \leq b,\\
     1- \frac{x-b}{\beta} & \quad b \leq x \leq b+ \beta,\\
     0  & \quad otherwise.
   \end{array}
  \right.
$$

\begin{figure}
\centering
\includegraphics[scale=1]{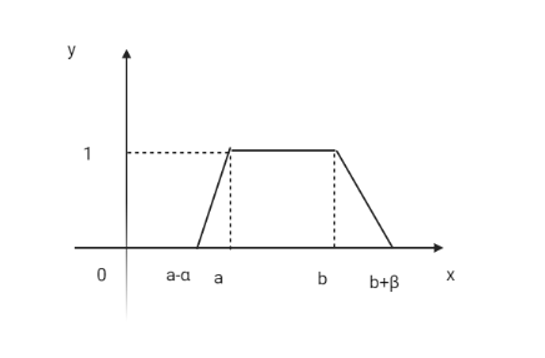}
\caption{Trapezoidal fuzzy number}
\label{fig:figura1}
\end{figure}

Assuming that $0<a-\alpha$ we have $D>0$ and the formula (5.1) becomes

(5.7) $E_{\lambda}(\frac{1}{D}) = \frac{\lambda}{\alpha}ln\frac{a}{a-\alpha} + \frac{1-\lambda}{\beta}ln\frac{b+\beta}{b}$

\end{remarca}

\begin{remarca} Assume that a triangular fuzzy number $D$ is written under the form $D = (a-\alpha,a,a+\beta)$, with $a\in \mathbb{R}$ and $\alpha,\beta\geq 0$. If $0<a-\alpha$ the formula (5.5) becomes

(5.8) $E_{\lambda}(\frac{1}{D}) = \frac{\lambda}{\alpha}ln\frac{a}{a-\alpha} + \frac{1-\lambda}{\beta}ln\frac{a+\beta}{a}$

\end{remarca}

The previous formulas (5.1), (5.5), (5.7) and (5.8) provide very computable expressions for the expected value $
E_{\lambda}(\frac{1}{D})$ for the particular cases when $D$ is a trapezoidal or
a triangular fuzzy number.

By using  these formulas  we are now able to compute the solution $x_1^{\ast},\cdots, x_n^{\ast}$ of
the optimization problem (4.6) whenever the components $D_1, \ldots, D_n$ of demand vector are trapezoidal fuzzy numbers, respectively triangular fuzzy numbers.

\begin{teorema}
Assume that the components $A_1, \ldots ,A_n$ of demand vector $\vec A$ are trapezoidal fuzzy numbers $D_i = (a_i - \alpha_i,a_i, b_i,b_i + \beta_i)$, $i = 1, \ldots, n$, where
$0< a_i - \alpha_i\leq a_i\leq b_i\leq b_i + \beta_i$, for $i = 1, \ldots, n$. Then the solution of the optimization problem (3.5) has the following form

(5.9) $x_i^\ast=\frac{d_i}{h_i[\frac{\lambda}{\alpha_i}ln\frac{a_i}{a_i-\alpha_i}+ \frac{1-\lambda}{\beta_i}ln\frac{b_i+\beta_i}{b_i}]}$, for all $i = 1, \ldots, n$.
\end{teorema}

\begin{proof}
By (5.7), for each $i = 1, \ldots, n$ we have

$E_{\lambda}(\frac{1}{D_i}) = \frac{\lambda}{\alpha_i}ln\frac{a_i}{a_i-\alpha_i} + \frac{1-\lambda}{\beta_i}ln\frac{b_i+\beta}{b_i}$

If we substitute these values of $E_{\lambda}(\frac{1}{D_1}),\cdots,E_{\lambda}(\frac{1}{D_n})$ in (4.8) then we get the desired formula (5.9).

\end{proof}

\begin{corolar}
If $D_1, \ldots, D_n$ are the triangular fuzzy numbers $D_i = (a_i- \alpha_i, a_i, a_i+\beta_i)$, $i = 1, \ldots, n$, where $0< a_i - \alpha_i\leq a_i\leq b_i + \beta_i$, for $i = 1, \ldots, n$ then the solution of optimization problem (4.3) has the form

(5.10) $x_i^\ast=\frac{d_i}{h_i[\frac{\lambda}{\alpha_i}ln\frac{a_i}{a_i-\alpha_i}+ \frac{1-\lambda}{\beta_i}ln\frac{a_i+\beta_i}{a_i}]}$, for all $i = 1, \ldots ,n$.

\end{corolar}

\begin{proof}
If in (5.9) one sets $b=a$, then the formula (5.10) follows immediately.
\end{proof}

Now we shall write the formula (5.9) for the following particular values of $\lambda$:

$(a)$ $\lambda=1/3$ (the pessimistic case)

(5.11) $x_i^\ast=\frac{3d_i}{h_i[\frac{1}{\alpha_i}ln\frac{a_i}{a_i-\alpha_i}+ \frac{2}{\beta_i}ln\frac{b_i+\beta_i}{b_i}]}$, for all $i = 1, \ldots ,n$.

$(b)$ $\lambda=1/2$ (the credibilistic case \cite{li1})

(5.12) $x_i^\ast=\frac{2d_i}{h_i[\frac{1}{\alpha_i}ln\frac{a_i}{a_i-\alpha_i}+ \frac{1}{\beta_i}ln\frac{b_i+\beta_i}{b_i}]}$, for all $i = 1, \ldots ,n$.

$(c)$ $\lambda=2/3$ (the optimistic case)

(5.13) $x_i^\ast=\frac{3d_i}{h_i[\frac{2}{\alpha_i}ln\frac{a_i}{a_i-\alpha_i}+ \frac{1}{\beta_i}ln\frac{b_i+\beta_i}{b_i}]}$, for all $i = 1, \ldots, n$.

\section{A numerical example}

In order to solve the optimization problems associated with some inventory models we should know the form of the variables $D_1,\cdots,D_n$ and of the
(probabilistic, credibilistic, possibilistic, etc.) indicators  that appear in models. In the examples of credibilistic inventory problems from \cite{li1}, \cite{li2}  the  expressions of $D_1,\cdots, D_n$ are assumed to be trapezoidal fuzzy numbers.

In general, the mathematical expressions of $D_1,\cdots,D_n$ are not known, but through measurements can be found different values of them. In the numerical example of possibilistic inventory problem from \cite{georgescu1} it started from a data table, then the method of Vercher et al. \cite{vercher}  was applied to determine the concrete form of fuzzy numbers $D_1,\cdots, D_n$.

In this section we will present the solution of an $m_{\lambda}$-inventory problem in which the initial information on the variables $D_1,\cdots,D_n$ (which in our case are trapezoidal numbers) is given in the form of a numerical table. In order to obtain the trapezoidal numbers that describe the demands $D_1,\cdots,D_n$ we will apply the sample percentile method of Vercher et al. \cite{vercher}.

We continue with the presentation of the values of basic parameters $d_i, c_i, d_i$, so the inventory problem is entirely defined. Finally, we apply the formulas (5.11)- (5.13)
in order to obtain the optimal solutions of the model.

Our inventory problem has a demand vector of size $10$. The following table contains the data we have on demand vector.

.

\begin{center}
 \begin{tabular}{||c c c c c c c c c c||}
 \hline
 Item1 & Item2 & Item3 & Item4 & Item5& Item6& Item7& Item8& Item9& Item10 \\ [0.5ex]
 \hline\hline
 35 & 20 & 30 & 25 & 28 & 33 & 18 & 18 & 31 & 20\\
 \hline
 30 & 30 & 50 & 25 & 32 & 37 & 27 & 28 & 33 & 27 \\
 \hline
 15 & 35 & 28 & 36 & 25 & 20 &  33 & 17 & 25 & 31\\
 \hline
 25 & 35 & 40 & 35 & 35 & 40 & 35 & 20 & 30 & 37\\
 \hline
 25 & 28 & 25 & 32 & 50 & 37 & 28 &  19 & 35 & 35\\
 \hline
 28 & 25 & 42 & 27 & 45 & 28 & 28 & 37 & 35 & 37\\
 \hline
 31 & 27 & 36 & 35 & 43 & 35 & 35 & 27 & 22 & 35\\
 \hline
 30 & 24 & 39 & 28 & 27 & 35 & 24 & 37 & 27 & 25\\
 \hline
 44 & 33 & 44 & 28 & 32 & 22 & 39 & 30 & 29 & 25\\
 \hline
 37 & 34 & 37 & 44 & 44 & 32 & 47 & 36 & 28 & 37\\
 \hline
 23 & 17 & 22 & 33 & 32 & 29 & 31 & 31 & 45 & 19\\ [1ex]
 \hline
\end{tabular}
\end{center}

In column $i$ of the table are placed the known values of item $i$. In a probabilistic inventory model, the above columns will contain values of random variables. In this case the maximization problem of the model will be obtained by usual statistical methods.

Under the hypothesis that the $10$ items are modeled by trapezoidal fuzzy numbers, one has to convert the data from the above table in $10$ such fuzzy numbers (each column is assigned to a trapezoidal fuzzy number).

Let's present shortly the percentile method of Vercher et al. \cite{vercher}, by which to a data set of real numbers $x_1, \ldots, x_m$ one assigns a trapezoidal fuzzy number $A=(a, b, \alpha, \beta)$.

Let us denote by $P_k$ the $k$-the percentile of the sample $x_1, \ldots, x_m$. Then the trapezoidal fuzzy number $A=(a, b, \alpha, \beta)$ will be determined by the formulas:

$a=P_{40}$, $b=P_{60}$, $\alpha=P_{40}-P_5$, $\beta=P_{95}-P_{60}$ (6.1)

By applying Vercher et al.'s method \cite{vercher} to each of the columns of the table one obtains the trapezoidal fuzzy numbers in the table below.

\begin{center}
 \begin{tabular}{||c c c c c ||}
 \hline
 $A_1$ & $A_2$ & $A_3$ & $A_4$ & $A_5$ \\ [0.5ex]
 \hline\hline
 (28,30,9,10.5) & (27,30,8.5,5) &(36,39,12.5, 5) & (28,32,3,4) & (32,35,6,10)  \\ [1ex]
 \hline
\end{tabular}
\end{center}

\begin{center}
 \begin{tabular}{||c c c c c ||}
 \hline
  $A_6$ & $A_7$ & $A_8$ & $A_9$ & $A_{10}$ \\ [0.5ex]
 \hline\hline
 (32,35,11,2) & (28,33,7,6) & (27,30,9.5,7) & (29,31,5.5,4) & (27,35,7.5,2) \\ [1ex]
 \hline
\end{tabular}
\end{center}

The trapezoidal fuzzy numbers $A_1, \ldots, A_{10}$ obtained from the initial table will be the components of the demand vector of a risk neutral multi--item inventory problem.
This inventory problem will be defined by the data in the first five columns of following table:

\begin{center}
 \begin{tabular}{||c c c c c c c c||}
 \hline
 Item & $d_i$ & $c_i$ & $h_i$ & $A_i=(a_i, b_i, \alpha_i, \beta_i)$ & $x_i^\ast (\lambda=1/3)$ & $x_i^\ast (\lambda=1/2)$ & $x_i^\ast (\lambda=2/3)$  \\ [0.5ex]
 \hline\hline
 1 & 12 & 2 & 0.5 & (28,30,9,10.5) & 718.21 & 668.76 & 627.44 \\
 \hline
 2 & 11 & 1 & 0.6 & (27,30,8.5,5) & 518.19 & 486.88 & 459.14 \\
 \hline
 3 & 14 & 3 & 0.5 & (36,39,12.5, 5) & 1019.75 & 961.42 & 909.4 \\
 \hline
 4 & 10 & 4 & 0.8 &  (28,32,3,4) & 387.92 & 371.9& 357.14 \\
 \hline
 5 & 11 & 5 & 0.9 & (32,35,6,10)  & 432.03 & 409.19 & 388.64\\
 \hline
 6 & 10 & 3 &  0.9 & (32,35,11,2) & 355.13 & 336.3& 319.37\\
 \hline
 7 & 12 & 2 & 0.5 & (28,33,7,6) & 743.93 & 696.25 & 654.32\\
 \hline
 8 & 15 & 1 & 0.6 & (27,30,9.5,7) & 710.45 & 661.32 & 618.54 \\
 \hline
 9 & 13 & 3 & 0.7 &  (29,31,5.5,4) & 563.24 & 541.63 & 521.61\\
 \hline
 10 & 13 & 4 & 0.9 & (27,35,7.5,2) & 437.88 & 405.88 & 378.24\\ [1ex]
 \hline
\end{tabular}
\end{center}

Columns two, three and four of the table contain the unit fixed costs, unit revenues and holding costs of the model.  The trapezoidal fuzzy numbers from the fifth column make up the demand vector in the $m_{\lambda}$--inventory problem. In fact, for distinct parameters $\lambda\in [0,1]$ we obtain distinct inventory problems. We consider the three inventory models (a)-(c) corresponding to the parameters $1/3$, $1/2$ and $2/3$.
By applying the formulas (5.11)-(5.13) we  obtain the  solutions of the three optimization problems. These solutions are placed in the last three columns of the above table.

\begin{remarca}
Regarding the last three columns of the table above, it is noticed that with the increase of the parameter $\lambda$ ($\frac{1}{3} < \frac{1}{2} < \frac{2}{3}$) the solution values  of the optimization problem decrease. The theoretical argument of this fact is given by
Proposition 8.2 in the Appendix.
\end{remarca}

\section{Conclusions}

In the work we studied a new inventory model whose construction is based on the parametric measure $m_{\lambda}$ (introduced by Yang and Iwamura in \cite{YI}) and on the notion of
$m_{\lambda}$--expected value (introduced by Dzouche et al. in \cite{dzuche1}). More precisely, in this inventory model, the demands and the total profit are fuzzy variables and the objective
function of the optimization problem is the $m_{\lambda}$--expected value of total profit. It was found the general form of the solution of the optimization problem and
when the demands are trapezoidal or triangular fuzzy variables computationally efficient forms of the solution have been found.

An open problem is finding the calculation formulas for the optimal solutions also when the demands are represented by other types of fuzzy variables: discrete repartitions, Erlang fuzzy variables, etc.

The inventory model in the paper is risk--neutral. Another open problem is the study of risk--averse inventory models in the framework of $m_{\lambda}$--theory.
It would also be interesting to treat some mean--value inventory model, in which besides maximizing the  $m_{\lambda}$--expected value of the
total profit to be required to minimize the $m_{\lambda}$--variance  of the total profit (the notion of $m_{\lambda}$--variance has been defined in \cite{dzuche1}).
Defining a notion of mean--absolute deviation in the context of an $m_{\lambda}$--theory would lead to an inventory model in which the risk is eventually represented by this indicator.

Continuing the research line from \cite{li1}, \cite{li2}, in paper \cite{li3} is investigated an inventory problem in which the components of the demand vector are type-$2$ fuzzy variables.
This model is studied with the techniques of Liu's credibility theory \cite{liu1}. It arises naturally a question of extending this model to $m_{\lambda}$--theory, so that giving the parameter $\lambda$ the value
$\frac{1}{2}$  to obtain as a particular case some results of \cite{li3}.

The newsvendor problem is a core concept in inventory management dealing with stochastic demand. Traditionally, it centers on a single goal: either minimizing expected costs or maximizing expected profits.

A mean--variance model for the newsvendor problem is presented in paper \cite{choi1}. A newsvendor problem is studied in which the maximization of expected profit and the minimization of risk, expressed by the profit variance, are required. It would be interesting to formulate and study a newsvendor problem in which the expected profit is expressed by  $m_\lambda$--expected value and the risk of profit by $m_\lambda$--variance (according to Definition 2 of \cite{dzuche1}).

\section{Appendix}

One asks  the question of how the solutions of the optimization problem (4.4) vary depending on the parameter $\lambda$. We will give a solution to this problem in case when
the demands $D_1, \ldots, D_n$ are trapezoidal fuzzy numbers.

\begin{lema}
Assume that $\xi$ is a trapezoidal fuzzy variable. If $\lambda_1 \leq \lambda_2$ then $E_{\lambda_1}(\xi) \leq E_{\lambda_2}(\xi)$.
\end{lema}

\begin{proof}
See Proposition 1 of \cite{dzuche1}.
\end{proof}

Let $\lambda_1, \lambda_2$ be two parameters in the interval [0, 1]. We consider the two inventory problems with the same input data, but with different objective functions of the optimization
problems, defined by the expected operators $E_{\lambda_1}(\xi)$ and $E_{\lambda_2}(\xi)$, respectively.

We denote by $x_1^\ast, \ldots, x_n^\ast$ the solution of the optimization problem corresponding to $E_{\lambda_1}(\xi)$ and with $y_1^\ast, \ldots, y_n^\ast$ the solution of the optimization problem corresponding to $E_{\lambda_2}(\xi)$.

\begin{propozitie}
Assume that the demands $D_1, \ldots, D_n$ are trapezoidal fuzzy variables. If $\lambda_1 \leq \lambda_2$ then $x_i^\ast \geq y_i^\ast$ for any $i=1, \ldots, n$.

\end{propozitie}

\begin{proof}
Assume that $\lambda_1 \leq \lambda_2$. By Proposition 4.4, solutions $x_1^\ast, \ldots, x_n^\ast$ and $y_1^\ast, \ldots, y_n^\ast$ are written in the following form:

(8.1) $x_i^\ast=\frac{d_i}{h_iE_{\lambda_1}(\frac{1}{D_i})}$ for $i=1, \ldots, n$.

(8.2) $y_i^\ast=\frac{d_i}{h_iE_{\lambda_2}(\frac{1}{D_i})}$ for $i=1, \ldots, n$.

Applying Lemma 8.1 for any $i=1, \ldots, n$ the following implications hold:

(8.3) $\lambda_1 \leq \lambda_2 \Rightarrow E_{\lambda_1}(\frac{1}{D_i}) \leq E_{\lambda_2}(\frac{1}{D_i}) \Rightarrow \frac{1}{E_{\lambda_2}(\frac{1}{D_i})} \leq \frac{1}{E_{\lambda_1}(\frac{1}{D_i})}$

By (8.1)-(8.3) for any $i=1, \ldots, n$ we will have:

$x_i^\ast-y_i^\ast=\frac{d_i}{h_i}(\frac{1}{E_{\lambda_2}(\frac{1}{D_i})} -\frac{1}{E_{\lambda_1}(\frac{1}{D_i})} )\geq 0$.

We conclude that $x_i^\ast \geq y_i^\ast$ for any $i=1, \ldots, n$.
\end{proof}

\section*{Acknowledgement}

The author would like to express her gratitude to the editor and the reviewers for their valuable suggestions in obtaining the final version of the paper.

\end{document}